\documentclass[12pt]{amsart}

\usepackage{amssymb}

\usepackage{graphicx,amssymb,amsfonts,epsfig,amsthm,a4,amsmath,url}

\usepackage[latin1]{inputenc}

\newtheorem{thm}{Theorem}[section]

\newtheorem{cor}[thm]{Corollary}

\newtheorem{lem}[thm]{Lemma}

\newtheorem{prop}[thm]{Proposition}

\theoremstyle{definition}

\newtheorem{defn}[thm]{Definition}

\newtheorem*{claa}{Claim}

\theoremstyle{remark}

\newtheorem{rem}[thm]{Remark}

\newtheorem{ex}[thm]{Example}

\numberwithin{equation}{section}

\newcommand{\Ac}{\mathcal{A}}
\newcommand{\Bc}{\mathcal{B}}
\newcommand{\Hc}{\mathcal{H}}

\newcommand{\Wc}{\mathcal{W}}

\newcommand{\F}{\mathbf{F}}

\newcommand{\Z}{\mathbf{Z}}

\newcommand{\N}{\mathbf{N}}

\newcommand{\R}{\mathbf{R}}

\newcommand{\Q}{\mathbf{Q}}

\newcommand{\Supp}{\textnormal{Supp}}

\newcommand{\SL}{\textnormal{SL}}

\newcommand{\A}{\mathcal{A}}

\begin{document}

\title[Proper actions of wreath products]{Proper actions of wreath products and generalizations}
\thanks{Y.C. and Y.S. have been supported by ANR project ``QuantiT'' (Nr JC08\textunderscore 318197).}

\author{Yves Cornulier}
\address{IRMAR \\ Campus de Beaulieu \\
35042 Rennes Cedex, France}
\email{yves.decornulier@univ-rennes1.fr}

\author{Yves Stalder}
\address{Clermont Universit\'e, Universit\'e Blaise Pascal, Laboratoire de  
Math\'ematiques, BP 10448, F-63000 Clermont-Ferrand, France\newline
CNRS, UMR 6620, Laboratoire de Math\'ematiques, F-63177 Aubi\`ere, France}
\email{yves.stalder@math.univ-bpclermont.fr}

\author{Alain Valette}
\address{Institut de Math\'ematiques, Universit\'e de Neuch\^atel, Rue \'Emile Argand
11, CP 158, 2009 Neuch\^atel, Switzerland}
\email{Alain.Valette@unine.ch}

\date{August 31, 2010}

\begin{abstract} We study stability properties of the Haagerup property and of coarse embeddability in a Hilbert space, under certain semidirect
products. In particular, we prove that they are stable under taking
standard wreath products. Our construction also provides a characterization of subsets with relative Property T in a standard wreath product.
\end{abstract}

\maketitle

\section{Introduction}

A countable group is \emph{Haagerup} if it admits a metrically proper isometric action on a Hilbert space. Groups with the Haagerup Property are also known, after Gromov, as a-T-menable groups as they generalize amenable groups. However, they include a wide variety of non-amenable groups: for example, free groups are a-T-menable; more generally, so are groups having a proper isometric action either on a CAT(0) cubical complex, e.g. any Coxeter group, or on a real or complex hyperbolic symmetric space, or on a product of several such spaces; this includes the Baumslag-Solitar group $\textnormal{BS}(p,q)$, which acts properly by isometries on the product of a tree and a real hyperbolic plane.

A nice feature about Haagerup groups is that they satisfy the strongest form of the Baum-Connes conjecture, namely the conjecture with coefficients \cite{HK}.

The Haagerup Property appears as an obstruction to Kazhdan's Property T and its weakenings like the relative Property T. Namely, a countable group $G$ has Kazhdan's Property T if every isometric action on a Hilbert space has bounded orbits; more generally, if $G$ is a countable group and $X$ a subset, the pair $(G,X)$ has the relative Property T if for every isometric action of $G$ on a Hilbert space, the ``$X$-orbit" of $0$, $\{g\cdot 0|g\in X\}$ is bounded. Clearly, if $X$ is infinite, then $G$ does not have the Haagerup Property.

The class of countable groups with Haagerup Property is obviously closed under taking subgroups. However, unlike the class of amenable groups, it is not closed under taking quotients, nor extensions, even semidirect products, as Kazhdan proved that the pair $(\Z^2\rtimes\SL_2(\Z),\Z^2)$ has the relative Property T.

However, for a semidirect product $\Gamma=W\rtimes G$ of groups with the Haagerup Property, it is reasonable to expect $\Gamma$ to have the Haagerup Property in the case when $G$ acts ``very freely'' on $W$. For instance, consider two Haagerup groups $G,\,H$, and let $W=\operatornamewithlimits{\mbox{\huge $\ast$}}_{g\in G} H$ be the free product of copies of $H$ indexed by $G$, with $G$ acting by shifting indices. Then $W\rtimes G$ is isomorphic to the free product $H\ast G$, and therefore it is Haagerup.

If instead of taking the free product of copies of $H$ indexed by $G$, we take the direct sum $H^{(G)}=\bigoplus_{g\in G}H$ and form the semidirect product $H^{(G)}\rtimes G$, where $G$ acts by shifting the copies of $H$ in the direct sum, we get the standard wreath product $H\wr G$ (often referred to as wreath product) of $H$ and $G$.

Our first main result, which answers Question 62.3 in \cite{Guido} is:

\begin{thm}
Let $G,H$ be countable groups. If $G$ and $H$ have the Haagerup Property, then so does the standard wreath product $H\wr G$.
\label{main}
\end{thm}

As obviously $H$ is Haagerup if and only if $H^{(G)}$ is, this appears as a stability under a special kind of extensions. This theorem was announced in \cite{CSV} and proved there in special cases, e.g. $E\wr \F$, where $E$ is finite and $\F$ is free. That special result appeared since then in the book~\cite{BrO}. We also refer to \cite{CSV} for applications of the result in harmonic analysis. We here mention another application. It was asked in \cite[7.7(1)]{Cth} whether the quotient of a Haagerup group by an amenable normal subgroup is always Haagerup; the answer is negative in a strong sense.

\begin{cor}
There exists a Haagerup group $\Gamma_1$ with a non-Haagerup quotient $\Gamma_2=\Gamma_1/N$ with $N$ normal abelian in $\Gamma_1$; namely $\Gamma_1=\Z\wr\SL_2(\Z)$ and $\Gamma_2=\Z^2\rtimes\SL_2(\Z)$.
\end{cor}
\begin{proof}
If $\Lambda=\SL_2(\Z)$ and $\Z\Lambda$ is its group ring, the standard wreath
product $\Z\wr\SL_2(\Z)$ can be identified with $\Z\Lambda \rtimes \Lambda$,
and $\Z^2$ is a cyclic $\Lambda$-module, so is a quotient of the free cyclic
$\Lambda$-module by some submodule $N$, and $\Z^2\rtimes\Lambda$ is the 
quotient of $\Z\Lambda \rtimes \Lambda$ by the abelian normal subgroup $N$.
\end{proof}

The definition of the Haagerup Property, in terms of affine isometric actions on Hilbert spaces, provides many examples, but is not very tractable to give stability results. Permanence of the class of countable Haagerup groups either under direct limits or under extensions with amenable quotients \cite[Example 6.1.6]{CCJJV}, relies on the characterization of the Haagerup Property in terms of unitary representations (see the beginning of Section \ref{sec:wr}).

The proof of Theorem \ref{main} is based on the characterization due to \cite{RS,CMV} of the Haagerup Property by actions on spaces with measured walls, see Section \ref{sec:walls}. These actions are more closely related to isometric actions on $L^1$ than on $L^2$. The fact that wreath products behave better for actions on $L^1$-spaces than Hilbert spaces is illustrated by the wreath product $\Z\wr\Z$: on the one hand it has an isometric action on an $L^1$-space whose orbital maps are quasi-isometric embeddings (as follows easily from the construction in Section \ref{sec:construction}); on the other hand it does not embed quasi-isometrically into a Hilbert space as follows from the existence of a quasi-isometrically embedded 3-regular tree in $\Z\wr\Z$ (which follows e.g. from \cite{CT}) and the non-existence of a quasi-isometric embedding of such a tree into a Hilbert space \cite{B}. Better constraints on embeddings of $\Z\wr\Z$ into a Hilbert space were brought out by \cite{AGS,ANP}.

Our second theorem, which follows from the same construction as the first one, concerns relative Property T. In the statement below, we identify the underlying set of the group $H\wr G$ with the cartesian product $H^{(G)}\times G$.

\begin{thm}\label{wreathT}
Let $H,G$ be countable groups and $X$ a subset of $H\wr G$. The following statements are equivalent:
\begin{itemize}\item[(i)]$(H\wr G,X)$ has relative Property T;\item[(ii)]there exist subsets $Y\subset H$ and $Z\subset G$ with $X\subset Y^{(Z)}\times Z$, the pairs $(H,Y)$ and $(G,Z)$ have relative Property T, and the following additional condition is satisfied: the function $X\rightarrow\N: (w,g)\mapsto \# \Supp(w)$ is bounded (where $\Supp(w)$ denotes the support of $w$).
\end{itemize}
\end{thm}

This additional condition shows in particular that whenever $G$ is infinite
and $H\neq \{1\}$, the wreath product $H\wr G$ does not have Property T, a
fact proved in \cite{CMV,Neu}. Also note that Condition (ii) can also be
rephrased as: there are relative Property~T subsets $Y\subset H=H^{\{1\}}$
and $Z\subset G$ such that $X\subset (YZ)^n$ for some $n\ge 1$.

Our construction also provides a coarse analog of Theorem \ref{main}.

\begin{thm}\label{coarse}
Let $G,H$ be countable groups. If $G$ and $H$ admit a coarse embedding into a
Hilbert space, then so does the standard wreath product $H\wr G$.
\end{thm}

Actually, Dadarlat and Guentner \cite{DG} proved that coarse embeddability into 
Hilbert spaces is preserved by extensions with exact quotient, which yields the
above theorem when $G$ is exact.

Since we made an earlier version of this text available, S.~Li \cite{Li} 
gave
alternative proofs (albeit based on similar ideas) of our Theorems 
\ref{main} and \ref{coarse} and 
proved a new
result: if $G$ and $H$ have positive (non-equivariant) Hilbert space
compression, then so does the wreath product $H\wr G$.

The above results concern standard wreath products. More generally, if $X$ is a
$G$-set, the permutational wreath product $H\wr_{X}G:=H^{(X)}\rtimes G$ can be
defined; the results of Section \ref{sec:other} apply only to some of these
permutational wreath products. For example, we prove

\begin{thm}
Let $G,H$ be countable groups, and $N$ a normal subgroup of $G$. Suppose that
$G$, $H$, and $G/N$ all have the Haagerup Property. Then so does the
permutational wreath product $H\wr_{G/N} G$.
\end{thm}

The non-trivial part of the converse, namely that if $H\wr_{G/N} G$ is Haagerup
and $H\neq\{1\}$, then $G/N$ is Haagerup, was proved by Chifan and Ioana
\cite{CI}, disproving a conjecture made in a preliminary version of this paper.

The outline of the paper is as follows. In Section \ref{sec:walls} we define our first main tool, namely measured wall structures, inspired from \cite{RS} and from the spaces with measured walls in \cite{CMV}.

In Section \ref{sec:gauges}, we introduce our second main tool, gauges, a generalization of the support function on $H^{(G)}$ (taking values in finite subsets of $G$), and give several examples. In Section \ref{sec:construction},
gauges are used to transfer invariant measured walls structures from $G$ to $W\rtimes G$. In Section \ref{sec:other} this allows us to prove the Haagerup property or coarse embeddability for several cases of semidirect products; in particular Theorems \ref{main} and \ref{coarse} are proved.

We take a closer look at wreath products in Section \ref{sec:wr}, where Theorem \ref{wreathT} is proved. Section \ref{sct:comp} deals with the $L^1$-compression of wreath products with free groups. We prove:

\begin{thm}\label{L1compr} Let $G$ be a finitely generated group and let $\F$ be a finitely generated free group. Then the equivariant $L^1$-compression of $G\wr \F$ is equal to the one of $G$.
\end{thm}

Finally, Appendix \ref{HSMeasure} shows that, on countable sets, the concept of measured walls structure introduced in Section \ref{sec:walls} below, is equivalent to the concept of space with measured walls from \cite{CMV}. Appendix \ref{AppHecke} links invariant proper kernels to Hecke pairs.
\medskip

{\bf Acknowledgments:} Thanks are due to Goulnara Arzhantseva, Bachir Bekka, 
Pierre-Emmanuel Caprace and Fr\'ed\'eric Haglund for their input at various 
stages of this work. We also thank the referee for his careful reading and useful corrections.


\section{Measured walls structures}\label{sec:walls}

Let $X$ be a set. Let $2^X$ be the power set of $X$, endowed with the product topology. For $x\in X$, denote $\A_x=\{A\subset X|A\ni x\}$: this is a clopen subset of $2^X$. Set $P'(X)=2^X-\{\emptyset,X\}$, a locally compact space in the relative topology.

\begin{defn}\label{MesWalStr}
A {\it measured walls structure} is a pair $(X,\mu)$ where $X$ is a set and $\mu$ is a Borel measure on $2^X$ such that for all $x,y\in X$
$$d_\mu(x,y):=\mu(\A_x\bigtriangleup\A_y)<\infty.$$
\end{defn}

It is then straightforward that $d_\mu$ is a pseudo-distance on $X$.

\begin{ex}\label{ex:trees}
Let $X$ be the vertex set of a tree $T$. Define a half-space as one class of the partition of $X$ into two classes, obtained by removing some edge from the tree. For a subset $B$ of $2^X$, define $\mu(B)$ as half the number of half-spaces contained in $B$. Then $(X,\mu)$ is a measured walls structure, and $d_\mu(x,y)$ is exactly the tree distance between the vertices $x$ and $y$.
\end{ex}

\begin{ex} \label{ex:discretewalls} More generally, let $(X,\Wc, f)$ be a discrete space with walls, meaning that $X$ is a set, $\Wc$ is a set of partitions into 2 classes (called walls), and $f:\Wc\rightarrow\N\cup\{\infty\}$ is a function such that, for every $x,y\in X$:
$$w(x,y):=\sum_{m\in\Wc(x|y)}f(m)<+\infty$$
where $\Wc(x|y)$ is the set of walls separating $x$ from $y$. Define a half-space as one class of some wall in $\Wc$; denote by $\Hc$ the set of half-spaces, and let $p:\Hc\to\Wc$ be the canonical map. For $B\subset 2^X$, define
$$\mu(B)=\frac{1}{2}\sum_{A\in B\cap\Hc} f\circ p(A).$$
Then $(X,\mu)$ is a measured walls structure, and $d_\mu(x,y)=w(x,y)$ for $x,y\in X$.

\end{ex}

\begin{lem}\label{Radon} For every measured walls structure $\mu$ on a countable set $X$, the restriction of $\mu$ to $P'(X)$ is a Radon measure.
\end{lem}

\begin{proof} Let $(X,\mu)$ be a measured walls structure. The locally compact space $P'(X)$ has a prebasis consisting of compact sets $(\mathcal{A}_x\setminus\mathcal{A}_y)_{x,y\in X}$, which have finite measure for $\mu$. Every compact subset $K$ of $P'(X)$ is contained in a finite union of subsets in the prebasis, so $\mu(K)<+\infty$. Since $X$ is countable, every open subset of $P'(X)$ is $\sigma$-compact, so every Borel measure on $P'(X)$ which is finite on compact subsets is a Radon measure (see Theorem 2.18 in \cite{Rudin}).
\end{proof}

Let $X,Y$ be sets; if $f:X\to Y$ is a map and $(Y,\mu)$ a measured walls structure, we can push forward the measure $\mu$ by the inverse image map $f^{-1}:2^Y\to 2^X$ to get a measured walls structure $(X,f^*\mu)$, whose distance is obviously $d_{f^*\mu}(x,x')=f^*d_\mu(x,x')=d_\mu(f(x),f(x'))$. When $f$ is one-to-one and is viewed as an inclusion, $f^*\mu$ can be viewed as the restriction of the measured walls structure $(Y,\mu)$ to $X$.

Let $(X_i,\mu_i)$ be a family of measured walls structures. Fix a base point in each $X_i$. Let $p_i$ denote the natural projection $\bigoplus X_j\to X_i$. Then the measure $\mu=\sum_ip_i^*\mu_i$ defines a measured walls structure whose associated pseudo-distance is $d_\mu((x_i),(y_i))=\sum_id_{\mu_i}(x_i,y_i)$. We call $(\bigoplus X_i,\mu)$ the {\it direct sum} of the measured walls structures $(X_i,\mu_i)$.

\begin{defn}
1) A {\it measure definite kernel} on $X$ is a function $\kappa:X\times X\to\R_+$ such that there exists a measured space $(Y,\mathcal{T},m)$ and a map $F:X\to \mathcal{T}$ with $\kappa(x,y)=m(F(x)\bigtriangleup F(y))$ for all $x,y\in X$.

2) For $p\in[1,\infty)$, an $L^p$-{\it embeddable kernel} on $X$ is a function $\kappa:X\times X\to\R_+$ such that there exists a measured space $(Y,\mathcal{T},m)$ and a map $f:X\to L^p(Y,m)$ with $\kappa(x,y)=\|f(x)-f(y)\|_p$ for all $x,y\in X$.
\end{defn}

\begin{prop}
Let $X$ be a countable set and let $\kappa:X\times X\to\R_+$ be a kernel. The following are equivalent.
\begin{itemize}
\item[(i)] $\kappa$ is measure definite;
\item[(ii)] $\kappa$ is an $L^1$-embeddable kernel;
\item[(ii')] $\kappa^{1/p}$ is an $L^p$-embeddable kernel for every $p\in [1,\infty)$.
\item[(iii)] $\kappa=d_\mu$ for some measured walls structure $(X,\mu)$.
\end{itemize}\label{md}
\end{prop}
\begin{proof}The less trivial implication (i)$\Rightarrow$(iii) is \cite[Proposition~1.2]{RS}; it is the only one for which we need the countability assumption (however we do not know if it is really needed). The implication (iii)$\Rightarrow$(i) is trivial: just map $x$ to $\A_x$.

The classical implication (ii)$\Rightarrow$(i) works as follows: observe that the class of measure definite kernels is a convex cone, closed under pointwise convergence \cite[Proposition~1.3]{RS}. As the cone of $L^1$-kernels is generated by cut metrics, i.e. pull-backs of the $\{0,1\}$-valued metric on two elements (see Section 4.2 in \cite{DL}), it is enough to check that cut metrics are measure definite, which is straightforward. The implication (ii')$\Rightarrow$(ii) is trivial, and we get (i)$\Rightarrow$(ii') by choosing any base-point $x_0\in X$, defining $f(x)=\mathbf{1}_{F(x)}-\mathbf{1}_{F(x_0)}\in L^p(Y,m)$, and observing that $\kappa(x,y)^{1/p}$ is the $L^p$-norm of $\mathbf{1}_{F(x)}-\mathbf{1}_{F(y)}=f(x)-f(y)$.
\end{proof}

\begin{prop}If $\kappa$ is an $L^2$-embeddable kernel, then it is an $L^1$-embeddable kernel.\label{L2L1}
\end{prop}
This just means that $L^2$ embeds isometrically into $L^1$. For the classical proof using gaussian random variables, see \cite{BDCK}; a proof using spaces with measured walls can be found in \cite[Proposition 1.4]{RS}.

\medskip

Let $G$ be a group; a kernel $\kappa$ on a $G$-set $X$ is uniform if for all $h_1,h_2
\in X$, the map $g\mapsto \kappa(gh_1,gh_2)$ is bounded on $G$.

On the other hand, for $X$ a $G$-set, we say that the measured walls structure
$(X,\mu)$ is {\it uniform} if the distance $d_\mu$ is a uniform kernel.
It is left-invariant if the measure $\mu$ is invariant for the left action of $G$ on $2^X$ (given by $g\cdot\mathbf{1}_Y=\mathbf{1}_{gY}$). Of course, left-invariant implies uniform.

\begin{prop}
Let $G$ be a group, $X$ a countable $G$-set, and let $\kappa$ be a left-$G$-invariant kernel on $X$.
Consider the following properties.

\begin{itemize}
\item[(i)] $\kappa$ is measure definite;
\item[(ii)] $(X,\kappa)$ is $G$-equivariantly embeddable into a metric space $Y$ with an isometric $G$-action, $Y$ being isometrically embeddable into some $L^1$-space;
\item[(ii')] for every $p\in [1,\infty)$, $(X,\kappa^{1/p})$ is $G$-equivariantly embeddable into a metric space $Y$ with an isometric $G$-action, $Y$ being isometrically embeddable into some $L^p$-space;
\item[(iii)] $\kappa=d_\mu$ for some $G$-invariant measured walls structure $(X,\mu)$;
\item[(iv)] $(X,\kappa)$ is $G$-equivariantly embeddable into some $L^1$-space with an isometric $G$-action;
\item[(iv')] for every $p\in [1,\infty)$, $(X,\kappa^{1/p})$ is $G$-equivariantly embeddable into some $L^p$-space with an isometric $G$-action.
\end{itemize}
Then$$\text{(iii)}\Leftrightarrow \text{(iv')}\Leftrightarrow\text{(iv)}\Rightarrow\text{(i)}\Leftrightarrow \text{(ii)}\Leftrightarrow\text{(ii')}.$$
Moreover:
\begin{itemize}
\item If $G$ is amenable, then all these properties are equivalent.
\item If $\kappa$ is an $L^2$-embeddable kernel, then it satisfies all the above conditions. This is in particular the case when $\kappa$ is the square root of a measure-definite kernel.
\end{itemize}
\label{equiv}
\end{prop}
\begin{rem}
Somewhat surprisingly, the reverse implication $(i)\Rightarrow(iv)$ is not known in general (for instance when $X=G$ is non-abelian free)\footnote{A ``proof'' of the implication $(i)\Rightarrow(iii)$ appears in Proposition 2(2) of \cite{CMV}: this proof is erroneous, as it rests on a misquotation of Theorem 2.1 in \cite{RS}. This error does not affect the main results in \cite{CMV}.}.
\end{rem}

\begin{proof}[Proof of Proposition \ref{equiv}]
The implication $(iii)\Rightarrow(iv')$ was essentially proved in \cite[Proposition~3.1]{CTV}. We recall here the proof, translated in the present language. Let $(X,\mu)$ be a $G$-invariant measured walls structure with $d_\mu=\kappa$. First consider the linear isometric action of $G$ on $L^p(2^X,\mu)$, given by $\pi_p(g)f(A)=f(g^{-1}A)$. Fix a base-point $x_0\in X$. Then $G$ acts affinely isometrically on $L^p(2^X,\mu)$ with linear part $\pi_p$, by $$\alpha_p(g)f=\pi_p(g)f+\mathbf{1}_{\A_{gx_0}}-\mathbf{1}_{\A_{x_0}}.$$
Now define the map $f:X\to L^p(2^X,\mu)$ by $f(x)=\mathbf{1}_{\A_{x}}-\mathbf{1}_{\A_{x_0}}$.

We see that $f(gx)=\mathbf{1}_{\A_{gx}}-\mathbf{1}_{\A_{x_0}}$ and $\pi_p(g)f(x)=\mathbf{1}_{\A_{gx}}-\mathbf{1}_{\A_{gx_0}}$. It follows that $\alpha_p(g)f(x)=f(gx)$, i.e. we get equivariance when $L^p(2^X,\mu)$ is a $G$-set through $\alpha_p$.

Also $$d(f(x),f(y))=\|\mathbf{1}_{\A_x}-\mathbf{1}_{\A_y}\|_p=\kappa(x,y)^{1/p}.$$

The implication $(iv)\Rightarrow(iii)$ follows essentially from \cite{CDH}. 
Namely, assume that $f:X \to Y=L^{1}(Z,m)$ is a $G$-equivariant embedding satisfying $\kappa(x,y) = ||f(x)-f(y)||_1$. Then $Y$ is a median space \cite[Example 2.8(7)]{CDH}. Thus, by \cite[Theorem 5.1]{CDH}, $Y$ can be turned into a space with measured walls $(Y,\Wc,\Bc,\mu)$, in the sense of \cite{CMV}, such that $G$ acts by automorphisms of spaces with walls \cite[Definition 3.5]{CDH}. We then consider the pull-back $(X,\Wc_X,\Bc_X,\mu_X)$ with respect to $f$; by \cite[Lemma 3.9]{CDH}, we have $\kappa=d_{\mu_X}$, and $G$ acts by automorphisms of spaces with walls. As $X$ is countable, this ensures the existence of a measure $\nu$ such that $(X,\nu)$ is a $G$-invariant measured walls structure with $\kappa = d_\nu$; see Appendix \ref{HSMeasure}.

All other implications are either trivial or follow from their non-equivariant counterpart; for $(i)\Rightarrow(ii')$, we set $f(x) = \mathbf 1_{F(x)}-\mathbf 1_{F(x_0)} \in L^p(Y,\mathcal T, m)$ and we
observe as in Proposition \ref{md} that $\kappa^{1/p} = ||f(x) - f(y)||_p$. Then, we define a $G$-action on $f(X)$ by $g\cdot f(x) = f(gx)$. To check that it is well-defined, observe that $f(x)=f(x')$ if and only if $\kappa(x,x') = 0$; hence, since $\kappa$ is $G$-invariant, $f(x)=f(x')$ implies $f(gx)=f(gx')$. Notice that \textit{a priori} $L^p(Y,\mathcal T, m)$ is endowed with no $G$-action, so that this does {\it not} prove (iv').

In case $G$ is amenable, let us now prove the implication $(i)\Rightarrow(iii)$. Suppose that $\kappa$ is measure definite and consider the set $\mathcal{C}$ of measures $\mu$ on $P'(X)$ such that $(X,\mu)$ is a measured wall structure ($\mu$ being extended to $2^X$ with $\mu(\{\emptyset,X\})=0$) with corresponding distance $d_\mu=\kappa$. By Proposition \ref{md}, $\mathcal{C}$ is non-empty. By Lemma \ref{Radon}, we can view each $\mu\in\mathcal{C}$ as a Radon measure on $P'(X)$.
We claim that $\mathcal{C}$ is weak-* compact. It is clearly closed and convex. Moreover, if $u$ is any continuous compactly supported function on $P'(X)$, then $\Supp\,u$ is contained in a finite union of subsets in the natural prebasis $(\mathcal{A}_x\setminus\mathcal{A}_y)_{x,y\in X}$ of the topology of $P'(X)$, showing that $\int_{P'(X)} ud\mu$ is bounded when $\mu$ ranges over $\mathcal{C}$. This shows that $\mathcal{C}$ is compact. Therefore the natural action of the amenable group $G$ has a fixed point $\nu$. Then $(X,\nu)$ is the desired $G$-invariant measured walls structure.

If $\kappa$ is an $L^2$-embeddable kernel (i.e. $\kappa^2$ is conditionally negative definite), arguing as in the proof of \cite[Theorem~2.1]{RS}, we get that $\kappa$ satisfies the strongest condition (iii). Finally, if $\kappa$ is measure definite, then $\kappa^{1/2}$ is $L^2$-embeddable, by the implication $(i)\Rightarrow (ii')$ in Proposition \ref{md}.
\end{proof}


\section{Gauges}\label{sec:gauges}

Let $W,X$ be sets. Let $\A=2^{(X)}$ denote the set of finite subsets of $X$.

\begin{defn}
An {\it $\A$-gauge} on $W$ is a function $\phi:W\times W\to \A$
such that:
$$\phi(w,w')=\phi(w',w)\quad\forall w,w'\in W;$$
$$\phi(w,w'')\subset\phi(w,w')\cup\phi(w',w'')\quad\forall w,w',w''\in W.$$

\end{defn}

\medskip

Observe that if $\phi_1,\phi_2$ are $\A$-gauges, then $\phi_1\cup\phi_2$ is an $\A$-gauge as well.

\begin{ex}
If $W=X$ is any set and $\A=2^{(X)}$, then $\phi(w,w')=\{w,w'\}$ is an $\A$-gauge on $X$.
\end{ex}

When $W$ is a group, $\phi$ is left-invariant if and only if it can be written as $\phi(w,w')=\psi(w^{-1}w')$.
For future reference, we record this as a definition.
\begin{defn} Let $W$ be a group. A {\it $W$-invariant $\A$-gauge} on $W$ is a function $\psi:W\to \A$ such that
$$\psi(w)=\psi(w^{-1})\quad\forall w\in W;$$
$$\psi(ww')\subset\psi(w)\cup\psi(w')\quad\forall w,w'\in W.$$
\label{G-gauge}\end{defn}

\begin{ex}\label{finitary}
If $W=\textnormal{Sym}_0(X)$, the group of finitely supported permutations of $X$, then $\psi(w)=\Supp(w)$ is a $W$-invariant $\A$-gauge on $W
$.
\end{ex}

\begin{ex}
If $F$ is a set with base point 1, $$W=F^{(X)}=\{f:X\to F|f(x)=1\textnormal{ for all but finitely many }x\},$$ then $$\phi(w,w')=\{x|w(x)\neq w'(x)\}$$ is an $\A$-gauge on $W$. If $F$ is a group, then this gauge is left-invariant.\label{gauge_support}
\end{ex}

\begin{ex}\label{regular} Suppose that $W$ is a group generated by the union of a family of subsets $(U_x)_{x\in X}$; we say that this family is {\it regular}
if it satisfies the following axiom: for any subset $Y$ of $X$, denoting by $W_Y$ the subgroup of $W$ generated by $\bigcup_{y\in Y}U_y$, we
assume that $W_Y\cap W_Z=W_{Y\cap Z}$ for all $Y,Z\subset X$. Then for each $w\in W$, we can define its support as the smallest finite subset $\psi(w)=Y$ of $X$ such that $w\in W_Y$. Plainly, the support function satisfies the axioms of a $W$-invariant $\Ac$-gauge on $W$.

Let us give several examples of groups with a regular generating family of subsets.

\begin{itemize}

\item[i)] A direct sum $\bigoplus H_x$, by the family $(H_x)$;

\item[ii)] A free product $\ast H_x$, by the family $(H_x)$;

\item[iii)] A Coxeter group, by the family of its given Coxeter generators (this family is regular by \cite[Chap.~IV.\S1, Theorem~2(ii)]{Bk});

\item[iv)] An Artin group, by the family of its given Artin generators (this family is regular by \cite[Theorem 4.14]{vdL});

\item[v)] Any group $G$ endowed with a Tits system (or BN-pair) $(B,N)$, with set of reflections $S$: the family $(BsB)_{s\in S}$ is regular; indeed, for $X\subset S$, let $P_X$ be the subgroup generated by $\bigcup_{s\in X}BsB$: by 3.2.2 in \cite{Tits}, these parabolic subgroups $P_X$ satisfy $P_{X\cap Y}=P_X\cap P_Y$.

\item[vi)] Let $V$ be a variety of groups, i.e. the class of all groups satisfying a set of laws $(w_i)_{i\in I}$. For $X$ a set, we denote by $V[X]$ the relatively free group on $X$, i.e. the largest quotient of the free group $\F[X]$ belonging to $V$. For $Y\subset X$, the inclusion $\F[Y]\hookrightarrow\F[X]$ induces an inclusion $V[Y]\hookrightarrow V[X]$. As a consequence, with $V_x=V[\{x\}]$, the generating family $(V_x)_{x\in X}$ is regular in $V[X]$. For the Burnside variety defined by the law $w^N$ (where $N\geq 2$), the relatively free group on $X$ is the free Burnside group $B(X,N)$ of exponent $N$. A variety $V$ is said to be locally finite if every finitely generated group in $V$ is finite. It is a well-known fact that the Burnside variety is locally finite for $N=2,3,4,6$.

\item[vii)] More generally, any group $W$ generated by the union of a family of subsets $(U_x)_{x\in X}$ such that for every $Y$ there exists a retraction $p_Y$ of $W$ onto $W_Y$ such that $p_Y|_{W_z}$ is constant for every $z\in X\backslash Y$. One example is the following. For every group $G$, denote by $G^*$ the largest residually finite quotient of $G$, i.e. the quotient of $G$ by the intersection of all finite index subgroups. For every set $X$, set $W=B(X,N)^*$, where $B(X,N)$ is the free Burnside group as above. For $x\in X$, let $U_x$ be the subgroup generated by $x\in X$. This family $(U_x)_{x\in X}$ satisfies the above assumption on retractions, by functoriality of the correspondence $G\to G^*$. Note that $W$ is locally finite, by the solution to the restricted Burnside problem due to Zelmanov \cite{Zel1,Zel2}.
\end{itemize}
\end{ex}

If $A,Y\subset X$, we say that $A$ {\it cuts} $Y$ and we write $A\vdash Y$ if neither $Y\subset A$ nor $Y\subset A^c$, i.e. $A$ induces a non-trivial partition of $Y$.

\begin{lem}
Let $Y,X$ be sets, $A\subset X$, and $\phi$ an $\A$-gauge on $Y$. Assume that $\phi(y,y)$ is a singleton for every $y\in Y$. Set $d_A(y,y')=1$ if $A\vdash\phi(y,y')$ and $0$ otherwise. Then $d_A$ is a pseudodistance on $Y$.\label{lemgap}
\end{lem}
\begin{proof}
The symmetry follows from the symmetry of $\phi$. Also, since $A$ cannot cut any singleton, $d_A$ vanishes on the diagonal. As $d_A$ takes values in $\{0,1\}$, the triangle inequality amounts to checking that being at distance zero is a transitive relation. Suppose that $d_A(y,y')=d_A(y',y'')=0$. Replacing $A$ by its complement if necessary, we can suppose that $\phi(y,y')\subset A$. In particular, $\phi(y,y')\cap\phi(y',y'')\subset A$. Moreover, $\phi(y',y')\subset\phi(y',y)\cup\phi(y,y')=\phi(y,y')$ and similarly $\phi(y',y')\subset \phi(y',y'')$. Therefore, $\phi(y,y')\cap\phi(y',y'')$ contains the singleton $\phi(y',y')$. In particular, $\phi(y',y'')\cap A\neq \emptyset$. As $d_A(y',y'')=0$, it follows that $\phi(y',y'')\subset A$. Therefore $\phi(y,y'')\subset\phi(y,y')\cup\phi(y',y'')\subset A$, so $d_A(y,y'')=0$.
\end{proof}

\section{Lifting measured walls structures}\label{sec:construction}

In this section, all sets are assumed to be (at most) countable.

\begin{defn} Let $\kappa$ be a symmetric, non-negative kernel on a set $X$.
\begin{enumerate}
 \item We say that $\kappa$ is \emph{proper} if, for every $x\in X,\,R\geq0$, the set $\{y\in X: \kappa(x,y)\leq R\}$ is finite;
 \item Let $Y$ be a subset of $X$. We say that $\kappa$ is \emph{proper on the subset $Y$} if the restriction of $\kappa$ on $Y$ is a proper kernel.
\end{enumerate}
\end{defn}

The element $(w,x)\in W\times X$ will be denoted by $wx$. The aim of this section is to prove:

\begin{thm}
Let $X,W$ be sets, $\A=2^{(X)}$. Let $\phi$ be an $\A$-gauge on $W$ and assume that $\phi(w,w)=\emptyset$ for all $w\in W$. Let $(X,\mu)$ be a
measured walls structure.

(i) There is a naturally defined measure $\tilde{\mu}$ on $2^{W\times X}$ such that $(W\times X,\tilde{\mu})$ is a measured walls structure with corresponding pseudo-distance given by
$$d_{\tilde{\mu}}(w_1x_1,w_2x_2)=\mu\{A|A\vdash\phi(w_1,w_2)\cup\{x_1,x_2\}\}.$$

(ii) Suppose that $X$ is a $G$-set, $W$ a $G$-group, and $\phi$ is $W$-invariant and $G$-equivariant. If $(X,\mu)$ is uniform (respectively invariant) under $G$, then $(W\times X,\tilde{\mu})$ is uniform (resp. invariant) under $W\rtimes G$.

(iii) Say that $\phi$ is {\it proper} if $\{w\in W|\phi(w,w')\subset F\}$ is finite whenever $F$ is finite and $w'\in W$.
If $d_\mu$ is proper and $\phi$ is proper, then $d_{\tilde{\mu}}$ is proper.
\label{cons}\end{thm}

\subsection{The $\{0,1\}$-valued pseudo-distance $d_A$}

Let us now consider two sets $X,W$, and $\A=2^{(X)}$. Let $\phi$ be an $\A$-gauge on $W$ and assume that $\phi(w,w)=\emptyset$ for all $w\in W$. Define an $\A$-gauge on $W\times X$
as follows: $\phi'(wx,w'x')=\{x,x'\}\cup\phi(w,w')$. Note that $\phi'(wx,wx)=\{x\}$. By Lemma \ref{lemgap}, for all $A\subset X$, the corresponding map $d_A$ is a pseudo-distance on $W\times X$.

Now, every $\{0,1\}$-valued pseudo-distance $d$ on a set $S$ defines a partition by a family of subsets $(B_i)$ of $S$, which are the classes of the equivalence relation $d(x,y)=0$. We can define a measured walls structure on $S$ as $\nu_d=\frac{1}{2}\sum_i\delta_{B_i}$ whose associated distance is obviously $d$.
Here we define $\nu_A$ as $\nu_{d_A}$.

\subsection{Integration of the $\nu_A$'s}\label{integ}

Suppose now that $(X,\mu)$ is a measured walls structure, and recall that 
$\phi$ takes values in finite subsets of $X$.
We want to define $\tilde{\mu}$ as the measure $``\int\nu_Ad\mu(A)"$.
For this to make sense, observe that $\nu_A$ is a Radon measure on 
$P'(W\times X)$. The locally compact topological space $P'(W\times X)$ has 
a prebasis consisting of compact subsets $K=\{B|w_1x_1\in B,w_2x_2\notin B\}$
for $w_1x_1,w_2x_2\in W\times X$.
For such a subset $K$, we have $\nu_A(K)=d_A(w_1x_1,w_2x_2)/2$. Moreover,
the map $A\mapsto d_A(w_1x_1,w_2x_2)$ is the characteristic function of the 
subset $\{A|A\vdash\phi(w_1,w_2)\cup\{x_1,x_2\}\}\subset 2^X$; in particular,
it is a measurable function. It follows that $A\mapsto \nu_A(K)$ is 
measurable, and
\[
\int_{2^X} \nu_A(K)d\mu(A)=\frac12\mu\{A|A\vdash\phi(w_1,w_2)\cup\{x_1,x_2\}\},
\]
which is always finite as $\phi(w_1,w_2)\cup\{x_1,x_2\}$ is finite. 

\begin{lem}For every continuous compactly supported function $f$ on 
$P'(W\times X)$, the function $A\mapsto \nu_A(f) := \int_{P'(W\times X)}
f\,d\nu_A$ is measurable on $2^X$ and the integral $\int_{2^X} \nu_A(f)
d\mu(A)$ is finite.
\end{lem}
\begin{proof} 
One can prove that $f$ is the uniform limit of functions $f_n$ which are 
linear combinations of characteristic functions of cylinders. Here, we call 
\emph{cylinder} a finite intersection of (at least one) elements of the 
prebasis. First, we claim that $A\mapsto \nu_A(C)$ is measurable for any 
cylinder $C$. Indeed, let us write $C=\bigcap_{j=1}^k K_j$, where 
$K_j=\{B|u_jx_j\in B, v_jy_j\notin B\}$ and $u_jx_j,v_jy_j\in W\times X$.
Fix $A\subset X$ and denote $W\times X = \bigsqcup_i B_i$ the partition 
induced by
the equivalence relation ``$d_A=0$''. If for some (necessarily unique) $i$, 
one has $u_jx_j\in B_i$ and $v_jy_j\notin B_i$ for all $j$, then $2\nu_A(C)=
\delta_{B_i}(C) + \sum_{\ell\neq i}\delta_{B_{\ell}}(C) = 1+0 = 1$; otherwise, 
one has $\nu_A(C)=0$. Hence, $A\mapsto 2\nu_A(C)$ is the characteristic 
function of the set 
\[
\{A\subset X | A \not \vdash \phi'(u_jx_j,u_{j'}x_{j'}) \text{ and } 
A\vdash\phi'(u_jx_j,v_{j}y_{j}) \text{ for all } j,j' \}
\]
(recall that $\phi'(wx,w'x')=\{x,x'\}\cup\phi(w,w')$); in particular, it is a
measurable function.

Then, the functions $A\mapsto\nu_A(f_n)$ are measurable as linear combinations of functions $A\mapsto \nu_A(C_i)$ associated to cylinders $C_i$, and the function $A\mapsto \nu_A(f)$ is measurable as the pointwise limit of functions $A\mapsto \nu_A(f_n)$.

Moreover, as the support of $f$ is contained in the union of finitely many
elements of the prebasis, say $\Supp(f)\subset\bigcup_{j=1}^n K_j$, we have 
\[
 \int_{2^X} \nu_A(f)d\mu(A) \le 
   \int_{2^X} \|f\|_\infty \sum_{j=1}^n\nu_A(K_i)d\mu(A) = 
   \|f\|_\infty \sum_{j=1}^n \int_{2^X} \nu_A(K_i)d\mu(A) \ .
\]
This proves that the integral $\int_{2^X} \nu_A(f)d\mu(A)$ is finite.
\end{proof}
It follows that the map $f\mapsto \int_{2^X}\nu_A(f)d\mu(A)$ is a positive linear form on $C_c(P'(W\times X))$, so by the Riesz representation theorem it is the integral with respect to some Radon measure $\tilde{\mu}$. If $K$ is the element of the prebasis defined above, we may evaluate on the characteristic function of $K$ (which is continuous) and the previous computation yields $\tilde{\mu}(K)=\frac12\mu\{A|A\vdash\phi(w_1,w_2)\cup\{x_1,x_2\}\}$, so that $(W\times X,\tilde{\mu})$ is a measured walls structure with associated distance
$$d_{\tilde{\mu}}(w_1x_1,w_2x_2)=\mu\{A|A\vdash\phi(w_1,w_2)\cup\{x_1,x_2\}\}.$$
This concludes the proof of part (i) of Theorem \ref{cons}.

\subsection{The case of $G$-sets}\label{Case-G-sets}

Now suppose that $X$ and $W$ are $G$-sets and that $\phi$ is $G$-equivariant, i.e. $$\phi(gw,gw')=g\phi(w,w')\quad \forall g\in G, w,w'\in W.$$ The set $W\times X$ is a $G$-set under the product action.

\begin{lem}
 For all $g\in G$, one has $\widetilde{g_*\mu}=g_*\tilde{\mu}$ and therefore
 $g\cdot d_{\tilde\mu}=d_{\widetilde{g_*\mu}}$.\label{tildeequiv}
\end{lem}

In particular if $\mu$ is invariant under $G$, then $\tilde{\mu}$ is $G$-invariant, and hence $d_{\tilde{\mu}}$ is $G$-invariant as well.

\begin{proof}
 The induced action on $C(P'(W\times X))$ is defined by the formula $g\cdot f (A) = f(g^{-1}\cdot A)$ and the measure $g_*\tilde\mu$ satisfies
 \[
  (g_*\tilde\mu)(f) = \tilde\mu(g^{-1}\cdot f) = \int_{2^X} \nu_A(g^{-1}\cdot f)d\mu(A)
 \]
 for any $f\in C_c(W\times X)$. On the other hand, setting $\varphi(A) = \nu_A(f)$, we obtain
 \begin{eqnarray*}
  \widetilde{g_* \mu}(f) & = & \int_{2^X} \nu_A(f)d(g_*\mu)(A) = (g_*\mu)(\varphi) , \quad \text{and} \\
  \int_{2^X} \nu_{gA}(f) d\mu(A) & = & \int_{2^X} (g^{-1}\cdot \varphi)(A) d\mu(A) = \mu(g^{-1}\cdot \varphi) \quad .
 \end{eqnarray*}
 Therefore, we have $\displaystyle\widetilde{g_* \mu}(f) = \int_{2^X} \nu_{gA}(f) d\mu(A)$. It remains to prove:
 \begin{claa}
  One has $\nu_{gA}(f) = \nu_A(g^{-1}\cdot f)$, for all $A\in 2^X$.
 \end{claa}
 The gauge $\phi$ being $G$-equivariant, one can easily prove the relation
 \[
  d_{gA}(g\cdot w_1x_1, g\cdot w_2x_2) = d_{A}(w_1x_1, w_2x_2) \quad ,
 \]
 for all $w_1x_1,w_2x_2 \in W \times X$. Hence, denoting by $(B_i)_i$ the partition of $W\times X$ induced by the pseudodistance $d_A$, the partition induced by $d_{gA}$ is $(gB_i)_i$.
 Consequently, one has
 
 $$2\nu_{gA}(f) = \sum_i\delta_{gB_i}(f) = \sum_i f(gB_i)$$ $$ = \sum_i (g^{-1}\cdot f)(B_i) = 2\nu_A(g^{-1}\cdot f).
 $$
 This proves the claim.
\end{proof}

\begin{lem}
If $\mu$ is uniform under $G$, then so is $\tilde{\mu}$.\label{tilduf}
\end{lem}
\begin{proof}
If $wx,w'x'\in W\times X$ and $g$ varies in $G$, we have
$$d_{\tilde\mu}(g\cdot wx,g\cdot w'x')=d_{\tilde\mu}(g\cdot w.g\cdot x,g\cdot w'.g\cdot x')$$
$$=\mu\{A|A\vdash\phi(gw,gw')\cup\{gx,gx'\}\}=\mu\{A|A\vdash gF\}$$
with $F=\phi(w,w')\cup\{x,x'\}$. Now

$$\mu\{A|A\vdash gF\}\le\sum_{y,z\in F}d_\mu(gy,gz)\le\sum_{y,z\in F}(\sup_{g\in G}d_\mu(gy,gz)),$$
which is finite.
\end{proof}

Suppose now moreover that $W$ is a $G$-group, so that $W\times X$ is a $(W\rtimes G)$-set with action $$wg\cdot (w',x)=(w.(g\cdot w'),g\cdot x),$$
and that $\phi$ is $W$-invariant, i.e. $\phi(ww',ww'')=\phi(w',w'')$ for all $w,w',w''\in W$. If $\mu$ is $G$-invariant, then one can see that $\tilde{\mu}$ and $d_{\tilde{\mu}}$ are $(W\rtimes G)$-invariant by using Lemma \ref{tildeequiv}.
For the same reason, if $\mu$ is assumed $G$-uniform instead of invariant, then it follows from Lemma \ref{tilduf} that $\tilde{\mu}$ is $(W\rtimes G)$-uniform as well. This proves part (ii) of Theorem \ref{cons}.

\subsection{Properness of $d_{\tilde{\mu}}$}\label{sec:properness}
We continue with the assumptions of Theorem \ref{cons}.

\begin{prop}
If a subset of $W\times X$ is $d_{\tilde{\mu}}$-bounded, then it is contained in
$$\{w\in W|\phi(w,w')\subset B\}\times B'$$
for some $d_\mu$-bounded subsets $B,B'$ of $X$ and $w'\in W$. If $d_\mu$ is proper,
the converse also holds.\label{properness}
\end{prop}
\begin{proof}
Fix $w'x'$ and $r>0$.
If $$d_{\tilde{\mu}}(wx,w'x')=\mu\{A|A\vdash\phi(w,w')\cup\{x,x'\}\}\le r,$$
then $\mu\{A|A\vdash\{v,x'\}\}\le r$ for every $v\in\phi(w,w')\cup\{x\}$, that is, $\phi(w,w')\cup\{x\}$ is contained in the $r$-ball around $x'$ for the distance $d_\mu$.

Conversely, assume that $d_\mu$ is proper and set $E= \{w\in W|\phi(w,w')\subset B\}\times B'$, for some $d_\mu$-bounded subsets $B,B'$ of $X$ and $w'\in W$. There exist $x'\in X$ and $R>0$ such that
$$
E\subset \big\{ wx\in W\times X | \phi(w,w') \cup\{x\} \subset \overline{B}(x',R) \big\} \ .
$$
For $wx\in E$, one has thus
$$
d_{\tilde{\mu}}(wx,w'x')=\mu\{A|A\vdash\phi(w,w')\cup\{x,x'\}\}\le \mu\{A|A\vdash\overline{B}(x',R)\} \ .
$$
If $A\vdash\overline{B}(x',R)$, then clearly $A\vdash\{y,x'\}$ for some $y\in\overline{B}(x',R)$. As $d_\mu$ is proper, we have $|\overline{B}(x',R)| =:n \in \N^*$, whence
$$
 d_{\tilde{\mu}}(wx,w'x') \le \sum_{y\in\overline{B}(x',R)}\mu\{A|A\vdash\{y,x'\}\} \le
   \sum_{y\in\overline{B}(x',R)} d_\mu(y,x') \le n\cdot R \,,
$$
so that $E$ is $d_{\tilde\mu}$-bounded.
\end{proof}

\begin{rem}
 The properness assumption on $d_\mu$ cannot be dropped. Indeed, endow $X=\F_\infty$ with the walls structure induced by the edges of its Cayley tree and $W = (\Z/2\Z)^{(X)}$ with the gauge
 $\phi(w,w')= \Supp(w^{-1}w')$. One can see that
 $$
 E = \big\{ w\in W| \phi(w,0_W) \subset \overline{B}(1_X,1) \big\} \times \{1_X\} \ ,
 $$
 where $1_X$ and $0_W$ are the trivial elements in $X$ and $W$,
 is $d_{\tilde\mu}$-unbounded, while $\overline{B}(1_X,1)$ is $d_\mu$-bounded.
\end{rem}

Proposition \ref{properness} immediately implies part (iii) of Theorem \ref{cons}, whose proof is therefore completed.

However, in applications (such as wreath products with infinite base group), the gauge may be non-proper. Let $u$ be a measure definite (resp. conditionally negative definite) kernel on $W$ and extend it to $W\times X$ by $u(wx,w'x')=u(w,w')$; 
this is still measure definite (resp. conditionally negative definite). When $W,X$ are $G$-sets and $u$ is $G$-invariant (resp. $G$-uniform), then this measure definite (resp. conditionally negative definite) kernel is $G$-invariant (resp. $G$-uniform) as well.

\begin{prop} Keep the notation as in Theorem \ref{cons}.
Suppose that $d_\mu$ is proper (but maybe not $\phi$). Let $u$ be a conditionally negative definite kernel on $W$, extended as above to $W\times X$. Suppose that $u$ is proper on every subset of the form
$\{w\in W|\phi(w',w)\subset F\}$
for some finite $F$ and $w'\in W$.
Then the conditionally negative definite kernel $d_{\tilde{\mu}}+u$ is proper on $W\times X$. If $W, X$ are $G$-sets, $\phi$ is $G$-equivariant, $\mu$ and $u$ are $G$-invariant (resp. $G$-uniform), then $d_{\tilde{\mu}}+u$ is $G$-invariant (resp. $G$-uniform).
\label{propercom}
\end{prop}
\begin{proof}
 Fix $w_0 x_0\in W\times X$ and $R>0$; we have to prove that the set
 \[
 E:=\{wx\in W\times X| d_{\tilde \mu}(w_0x_0,wx) + u(w_0,w) \leq R\}
 \]
 is finite. Elements $wx\in E$ satisfy the relations $d_{\tilde \mu}(w_0x_0,wx) \leq R$ and ${u(w_0,w) \leq R}$. As $E$ is a $d_{\tilde \mu}$-bounded set, Proposition \ref{properness} gives a $d_\mu$-bounded set $B\subset X$ and a point $w'\in W$ such that one has
 \[
 E\subset \big\{w\in W| \phi(w',w)\subset B \big\} \times B \ .
 \]
 Moreover, $B$ is finite, since $d_\mu$ is proper. Afterwards, by properness of $u$ on the subset $\{w\in W| \phi(w',w)\subset B\}$, we obtain that $E$ is finite.

 The last assertion follows directly from the above remarks and the discussion in Paragraph \ref{Case-G-sets}.
\end{proof}

\section{Applications}\label{sec:other}

\subsection{Haagerup Property}\label{sec:AppliHaagerup}
Recall that a group is Haagerup if 1 is a pointwise limit of $C_0$ positive definite functions. A group is Haagerup if and only if all its finitely generated subgroups are Haagerup \cite[Proposition 6.1.1]{CCJJV}. Moreover, a countable group is Haagerup if and only if it has a proper conditionally negative definite function, which is also equivalent to having a proper measure definite function (this follows from Proposition \ref{L2L1} and the implication (i)$\Rightarrow$(ii') of Proposition \ref{md}).

\begin{thm}\label{second} Let $W, G$ be groups, with $G$ acting on $W$ by automorphisms. Set $\A=2^{(G)}$. Let $\psi$ be a left $W$-invariant, $G$-equivariant $\A$-gauge on $W$, in the sense of Definition \ref{G-gauge}. Assume that there exists a $G$-invariant conditionally negative definite function $u$ on $W$ such that, for every finite subset $F\subset G$, the restriction of $u$ to every subset of the form $W_F:=\{w\in W:\psi(w)\subset F\}$ is proper. Then $W\rtimes G$ is Haagerup if and only if $G$ is Haagerup.
\end{thm}

\begin{rem} The assumption on $u$ in Theorem \ref{second}, implies that $W$ is Haagerup. To see it, we prove that every finitely generated subgroup of $W$ is Haagerup. Observe that, for $F$ a finite subset of $G$, the set $W_F$ is a subgroup of $W$. If $w_1,...,w_n$ is a finite collection of elements in $W$, take $F=\bigcup_{i=1}^n \psi(w_i)$, so that the subgroup $\langle w_1,...,w_n\rangle$ generated by the $w_i$'s is contained in $W_F$. The assumption on $u$ then implies that $\langle w_1,...,w_n\rangle$ is Haagerup. By the remarks preceding Theorem \ref{second}, the group $W$ is itself Haagerup.
\end{rem}
\begin{proof}[Proof of Theorem \ref{second}] In view of the remarks preceding Theorem \ref{second}, we may assume that $W$ and $G$ are countable. We can suppose that $\psi(1)=\emptyset$ (otherwise we just change $\psi(1)$ by $\emptyset$, without altering the other hypotheses).
Since the Haagerup property is inherited by subgroups, one implication is obvious. To prove the non-trivial implication, we first extend $u$ to $W\rtimes G$ by setting $u(wg):=u(w)$. As $u$ is conditionally negative definite and $G$-invariant on $W$, the extension is conditionally negative definite on $W\rtimes G$. On the other hand, as $G$ is Haagerup, there exists a proper, $G$-invariant, $L^2$-embeddable kernel $\kappa$ on $G$, so by Proposition \ref{equiv} there exists a $G$-invariant measured walls structure $(G,\mu)$ such that $\kappa=d_\mu$.
Then we consider the function $wg\mapsto d_{\tilde{\mu}}(1,wg)+ u(wg)$: by Theorem \ref{cons} and Proposition \ref{propercom}, this is a proper conditionally negative definite function on $W\rtimes G$, so that $W\rtimes G$ is Haagerup.
\end{proof}

\begin{proof}[Proof of Theorem \ref{main}] Let $H,G$ be countable Haagerup groups. Set $W=H^{(G)}$, so that $H\wr G=W\rtimes G$. The $W$-invariant gauge $\psi(w)=\Supp(w)$ is $G$-equivariant; let $u$ be a proper, conditionally negative definite function on $H$; extend $u$ to $W$ by $u(w)=\sum_{g\in G}u(w_g)$: the extended function is conditionally negative definite, and its  restriction to $H^S$ is proper, for every finite subset $S\subset G$. So Theorem~\ref{second} applies.
\end{proof}

\begin{cor}\label{locallyfinite}  Let $W, G$ be groups, with $G$ acting on $W$ by automorphisms. Set $\A=2^{(G)}$. Let $\psi$ be a left $W$-invariant, $G$-equivariant $\A$-gauge on $W$, in the sense of Definition \ref{G-gauge}. Assume that there exists a $G$-invariant conditionally negative definite function $u$ on $W$ such that, for every finite subset $F\subset G$, the subgroup $W_F$ is finite. Then $W\rtimes G$ is Haagerup if and only if $G$ is Haagerup.
\end{cor}

\begin{proof} Since $W_F$ is finite, the function $u\equiv 0$ satisfies the assumption in Theorem \ref{second}, which therefore applies to give the result.
\end{proof}

\begin{ex} Let $X=G$ be a countable group. Define the group $W$ either as in Example \ref{finitary} or as in (vi) (locally finite variety) or (vii) (restricted Burnside groups) of Example \ref{regular}: in each case, the subgroups $W_F$ are finite (for $F\in 2^{(G)}$), so that $W\rtimes G$ is Haagerup if and only if $G$ is Haagerup, by Corollary \ref{locallyfinite}.
\end{ex}

\begin{ex} We elaborate on (iii) of Example \ref{regular}. Let $G$ be a countable group, endowed with a left invariant structure of Coxeter graph $X$. Let $W_X$ be the Coxeter group  associated with $X$, so that $G$ acts by automorphisms on $W_X$. Let $u$ be the word length on $W_X$ with respect to the set of Coxeter generators. Then $u$ is conditionally negative definite \cite{BJS} and, for $F\in 2^{(G)}$, the restriction $u|_{W_F}$ is proper, as it corresponds to word length on a finitely generated group. So Theorem \ref{second} applies: $W\rtimes G$ is Haagerup if and only if $G$ is.
\end{ex}

\subsection{Coarse embeddability}
We freely use basic terminology about coarse geometry, for which we refer to \cite{R}. Pseudometric spaces are endowed with their bounded coarse structure and groups with their left `canonical' coarse structure. In other words:
\begin{enumerate}
 \item if $(X,d)$ is a pseudometric space, a subset $E$ of $X\times X$ is controlled if and only if $\sup\{d(x,y)|(x,y)\in E\}<+\infty$;
 \item if $G$ is a group, a subset $E$ of $G\times G$ is controlled if and only if the subset $\{g^{-1}h|(g,h)\in E\}$ is finite.
\end{enumerate}
In case a group $G$ is endowed with a left-invariant proper metric $d$, the coarse structures arising from (1) and (2) coincide.
\begin{defn}
 Let $\kappa$ be a symmetric, non-negative, kernel on a group $G$. It is \emph{effectively proper} if, for all $R>0$, the subset $\{(x,y)\in X\times X: \kappa(x,y) \leq R\}$ is controlled.
\end{defn}
There is a strong link with coarse embeddability into Hilbert spaces.
\begin{prop}\label{kernelHilbert}
 Let $G$ be a group. The following are equivalent:
 \begin{enumerate}
 \item[(i)] $G$ coarsely embeds into a Hilbert space;
 \item[(ii)] there exists a conditionally negative definite, $G$-uniform, and effectively proper kernel on $G$.
\end{enumerate}
\end{prop}
This proposition is interesting for countable groups only: it is easy to see
that any group satisfying (i) or (ii) is countable.
\begin{proof}
If $f$ is a coarse embedding of $G$ into a Hilbert space, then the kernel $\kappa(x,y) = \|f(x)-f(y)\|^2$ satisfies (ii). Conversely, if $\kappa$ is a kernel as in (ii), the GNS-construction furnishes a coarse embedding of $G$ into a Hilbert space. Details are left to the reader.
\end{proof}

We now adapt ideas of Sections \ref{sec:properness} and \ref{sec:AppliHaagerup} to prove Theorem \ref{coarse}.

\begin{prop}\label{controlled}
Keep the notation as in Theorem \ref{cons} (i). Then, the set
\[
 \bigcup_{(wx,w'x')\in E}\left( \phi(w,w') \cup \{x,x'\} \right)^2 \subset X\times X
\]
is $d_\mu$-controlled, whenever $E$ is a $d_{\tilde{\mu}}$-controlled subset of $(W\times X)^2$.
\end{prop}
\begin{proof}
Take a positive number $R$ such that $d_{\tilde \mu}(wx,w'x')\leq R$ for all $(wx,w'x')\in E$. Then, for all $u,v \in X$ and $(wx,w'x')\in E$ such that $u,v \in \phi(w,w') \cup \{x,x'\}$, we have
\[
 d_\mu(u,v) \le \mu\{A|A\vdash\phi(w,w')\cup\{x,x'\}\} = d_{\tilde\mu}(wx,w'x')\le R \ . \qedhere
\]
\end{proof}

\begin{prop}\label{propercontrolled}
Let $G$ be a group and keep the notation as in Theorem \ref{cons} (i), with $X=G$. Let $u$ be a conditionally negative definite kernel on $W$, extended to $W\times G$ by $u(wx,w'x')=u(w,w')$. Suppose that $W$ is a $G$-group, $d_\mu$ is effectively proper, and
\begin{enumerate}
 \item[(H)] for all controlled subsets $C\subset G\times G$ and for all $R>0$, the subset
 \[
  \big\{(vx,wy)\in (W\rtimes G)^2 \, | \, (\phi(v,w)\cup\{x,y\})^2 \subset C \text{ and } u(v,w)\le R \big\}
 \]
 is controlled.
\end{enumerate}
Then, $d_ {\tilde \mu} + u$ is an effectively proper kernel on $W\rtimes G$.
\end{prop}

\begin{proof}
 Fix $R>0$; we have to prove that the set
 \[
 E:=\{(vx,wy)\in (W\rtimes G)^2 \, | \,  d_{\tilde \mu}(vx,wy) + u(v,w) \leq R\}
 \]
 is  controlled. As $E$ is $d_{\tilde\mu}$-controlled, Proposition \ref{controlled} tells us that the subset
 \[
  D:=\bigcup_{(wx,w'x')\in E}\left( \phi(w,w') \cup \{x,x'\} \right)^2
 \]
 is $d_\mu$-controlled. As $d_\mu$ is  effectively proper, $D$ is also controlled with respect to the group coarse structure. Finally, $E$ is contained in
 \[
  \big\{(vx,wy)\in (W\rtimes G)^2 \, | \, (\phi(v,w)\cup\{x,y\})^2 \subset D \text{ and } u(v,w)\le R \big\} \ ,
 \]
 so that it is a controlled subset of $(W\rtimes G)^2$ by (H).
\end{proof}

Here is a coarse analog of Theorem \ref{second}.

\begin{thm}\label{coarsegen}
 Let $W, G$ be groups, with $G$ acting on $W$ by automorphisms. Set $\A=2^{(G)}$. Let $\psi$ be a left $W$-invariant, $G$-equivariant $\A$-gauge on $W$, in the sense of Definition \ref{G-gauge}. Assume that there exists a $W$-uniform, $G$-invariant, conditionally negative definite kernel $u$ on $W$ such that, for every finite subset $F\subset G$ and every $R>0$, the set
 \[
  \{(v,w)\in W\times W| \ \psi(v^{-1}w) \subset F \text{ and } u(v,w)\le R\}
 \]
 is controlled.
 Then $W\rtimes G$ coarsely embeds into a Hilbert space if and only if $G$ does.
\end{thm}

\begin{proof}
 The ``only if'' part is obvious. Assume that $G$ coarsely embeds into a Hilbert space.
 By Proposition \ref{kernelHilbert}, there exists a conditionally negative definite kernel $\kappa$ on $G$ which is $G$-uniform and effectively proper. By Proposition \ref{md}, there exists a measured walls structure $(G,\mu)$ such that $\kappa=d_\mu$. We extend $u$ to $W\rtimes G$ by $u(vg,wh) = u(v,w)$. This kernel is still conditionally negative definite, and is $(W\rtimes G)$-uniform.

 By Theorem \ref{cons} (ii), the kernel $d_{\tilde \mu} + u$ on $W\rtimes G$ is $(W\rtimes G)$-uniform; it is also conditionally negative definite by construction. Now, by Proposition \ref{kernelHilbert}, it is sufficient to prove that $d_{\tilde \mu} + u$ is effectively proper.

 In order to apply Proposition \ref{propercontrolled}, let us check hypothesis (H). Let $C\subset G\times G$ be a controlled subset and let $R>0$. Up to enlarging $C$, we may assume that $C$ is symmetric and invariant under the diagonal action. We are to show that
 \[
  E:=\{(vg,wh)\in (W\rtimes G)^2 | \, (\psi(v^{-1}w)\cup \{g,h\})^2 \subset C \text{ and } u(v,w)\le R\}
 \]
 is a controlled subset. As $C$ is controlled, the set $F':=\{x^{-1}y| \, (x,y)\in C\}$ is finite. Suppose $(vg,wh)\in E$. We get
 \[
 \big(\psi(g^{-1}v^{-1}wg)\cup \{1,g^{-1}h\}\big)^2 \subset C
 \]
 by $G$-invariance of $C$, so that $g^{-1}h\in F'$ and $\psi(g^{-1}v^{-1}wg)\subset F'$. Now, by hypothesis, the set
 $
  \{(s,t)\in W\times W| \psi(s^{-1}t) \subset F' \text{ and } u(s,t)\le R\}
 $
 is controlled, so that
 \[
   F'':=\{s^{-1}t|\, s,t\in W, \,\psi(s^{-1}t) \subset F' \text{ and } u(s,t)\le R\}
 \]
 is finite. Consequently, we obtain
 $
  (vg)^{-1}wh = (g^{-1}v^{-1}wg)g^{-1}h \in F''F' \ .
 $
 The last set being finite, we have proved that $E$ is controlled. Finally, by Proposition \ref{propercontrolled}, the kernel $d_{\tilde\mu} + u$ is effectively proper, as desired.
\end{proof}

\begin{proof}[Proof of Theorem \ref{coarse}]
 By Proposition \ref{kernelHilbert}, there exists a conditionally negative definite, $H$-uniform, and effectively proper kernel $u$ on $H$. We extend $u$ to $W:=H^{(G)}$ by $u(v,w) = \sum_{g\in G}u(v_g,w_g)$. It is easy to check that this kernel on $W$ is conditionally negative definite, $G$-invariant, and $(W\rtimes G)$-uniform.

 Whenever $F$ is a finite subset of $G$ and $R$ is a positive number, the set
 \[
  E:= \{(v,w)\in W\times W| \Supp(v^{-1}w) \subset F \text{ and } u(v,w)\le R\}
 \]
 is controlled. Indeed, if $(v,w)\in E$, we have $v_x=w_x$ for $x\in G\setminus F$ and ${u(v_x,w_x)\le R}$ for $x\in F$. As $u$ is effectively proper on $H$, the set $F':=\{a^{-1}b| \, a,b\in H,\, u(a,b)\le R\}$ is finite. Hence, we get $v_x^{-1}w_x \in F'$ for all $x\in F$, so that $\{v^{-1}w|(v,w)\in E\}$ is finite.

 Finally, Theorem \ref{coarsegen} implies that $H\wr G = W\rtimes G$ coarsely embeds into a Hilbert space.
\end{proof}

\section{Wreath products}\label{sec:wr}
\subsection{Permutational wreath products}

Let $L$ be a subgroup of $G$; we say that $L$ is {\it co-Haagerup} in $G$ if there exists a proper $G$-invariant conditionally negative definite kernel on $G/L$.

\begin{ex}
\begin{itemize}
\item[i)] Obvious examples include: finite index subgroups; finite subgroups in a Haagerup group; normal subgroups such that $G/L$ is Haagerup.
\item[ii)] A subgroup $H$ of $G$ is {\it co-F\o lner} if the homogeneous space $G/H$ carries a $G$-invariant mean. It should be emphasized that co-F\o lner subgroups are not necessarily co-Haagerup. To see it, consider $G=(\Z \oplus \Z)\rtimes (\Z/2\Z)$ (where $\Z/2\Z$ acts by the flip $\sigma$ of the two factors) and $H$ the first factor in $\Z\oplus\Z$. Clearly $H$ is co-F\o lner in $G$. But observe that every length function $\ell:G\to\R^+$ which is bounded on $H$ is bounded on $G$; indeed, by sub-additivity $\ell$ is bounded first on $\sigma H\sigma^{-1}$ (i.e. the second factor of $\Z\oplus\Z$), then on $\Z\oplus\Z$ (which is generated by $H$ and $\sigma H\sigma^{-1}$), then on $G$. As a consequence, $H$ is not co-Haagerup in $G$.
\item[iii)] If $L$ is co-Haagerup in $G$, then $(G,L)$ is a Hecke pair (this follows e.g. from Proposition \ref{CarHecke} in Appendix B). The converse is not true, as illustrated by the Hecke pair $(\SL_3(\Z[1/p]),\SL_3(\Z))$.
\item[iv)] If $G$ acts on a locally finite tree $T$ and $L$ is some vertex stabilizer, then the distance function on $T$ descends to a proper $G$-invariant conditionally negative definite kernel on $G/L$, i.e. $L$ is co-Haagerup in $G$. This applies for instance to $L=\SL_2(\Z)$ in $G=\SL_2(\Z[1/p])$ (use the $\SL_2(\Q_p)$-action on the $(p+1)$-regular tree). Another example is given by a subgroup $L$ in an HNN-extension $G=\textnormal{HNN}(L,A,\vartheta)$, where both $A$ and $\vartheta(A)$ have finite index in $L$. For a concrete example, consider e.g. the Baumslag-Solitar group
    $$G=\textnormal{BS}(m,n)=\langle a,b\,|\,ab^ma^{-1}=a^n\rangle$$
    with respect to the subgroup $L=\langle b\rangle$.
 \end{itemize}
\end{ex}

\begin{thm}\label{corhagper}
Let $H,G$ be groups, and let $L$ be a subgroup of $G$. Suppose that $G$ and $H$ are Haagerup and that $L$ is co-Haagerup in $G$. Then the permutational wreath product $H\wr_{G/L}G=H^{(G/L)}\rtimes G$ is Haagerup.
\end{thm}

In the proof of Theorem \ref{corhagper}, we use the construction of Section \ref{sec:construction} and the following auxiliary construction.

\begin{lem}Let $G$ be a group, $Y$ a $G$-set, and $H$ another set. Let $(H,\sigma)$ be a measured walls structure. Then there exists a naturally defined $G$-invariant measure $\hat{\sigma}$ on $2^{H^{(Y)}\times G}$, such that $(H^{(Y)}\times G,\hat{\sigma})$ is a measured walls structure with associated distance given by
$$d_{\hat{\sigma}}(wg,w'g')=\sum_{y\in Y}d_\sigma(w_y,w'_y).$$ Suppose moreover that $H$ is a group. If $(H,\sigma)$ is left-invariant (resp. uniform) under $H$, then $(H^{(Y)}\rtimes G,\hat{\sigma})$ is invariant (resp. uniform) under $H^{(Y)}\rtimes G$.\label{aux}
\end{lem}
\begin{proof}
Consider the direct sum, indexed by $Y$ of copies of the measured walls structure $(H,\sigma)$, as in Section \ref{sec:walls}. The corresponding pseudo-distance is given by $d(w,w')=\sum_{y\in Y}d_\sigma(w_y,w'_y)$. Take the inverse image to $H^{(Y)}\times G$, to get the $G$-invariant measured walls structure $(H^{(Y)}\rtimes G,\hat{\sigma})$, whose associated pseudo-distance is the desired one. The last assertion is also straightforward.
\end{proof}

\begin{proof}[Proof of Theorem \ref{corhagper}]
Using the remarks at the beginning of Section \ref{sec:other}, we can suppose that $G$ and $H$ are countable.

Define $X=G/L$, $\A=2^{(X)}$, and $W=H^{(G/L)}$. Let $\phi$ be the $\A$-gauge on $W$ given by Example \ref{gauge_support}: $\phi(w,w')=\Supp(w^{-1}w')$.

By the co-Haagerup Property, there exists a $G$-invariant $L^2$-embeddable proper kernel $\kappa$ on $G/L$. By Proposition \ref{equiv}, there exists a $G$-invariant measured walls structure $\mu$ on $X$ with $d_\mu=\kappa$.

Consider the measured walls structure $(W\times X,\tilde{\mu})$ constructed in Section \ref{sec:construction}. By Theorem \ref{cons}, the corresponding pseudo-distance is given by
$$d_{\tilde{\mu}}(1,wx)=\mu\{A|A\vdash\phi(1,w)\cup\{1,x\}\}.$$

Let us now consider the projections $\rho:W\rtimes G\to W\times X$ and $p:W\rtimes G\to G$.

By properness of $d_\mu$, it then follows from Proposition \ref{properness} that any subset of $W\times G$ which is bounded for $d_{\rho^*\tilde{\mu}}$ is contained in a subset of the form
$\{wg\in W\times G\,|\,\Supp(w)\subset F,g\in F'\}$, where $F\subset G/L$ is finite and $F'$ is a subset of $G$ with finite image in $G/L$.

On the other hand, let $(G,\lambda)$ be a proper $G$-invariant measured walls structure. The bounded subsets of $W\rtimes G$ for $d_{p^*\lambda}$ are contained in subsets of the form $W\times F$ for $F\subset G$ a finite subset. It follows that bounded subsets for the pseudo-distance associated to the measured walls structure $(W\rtimes G,\rho^*\tilde\mu+p^*\lambda)$, which is $W\rtimes G$-invariant, are contained in subsets of the form $\{wg\in W\times G\,|\,\Supp(w)\subset F,g\in F'\}$, where $F\subset G/L$ and $F'\subset G$ are both finite.

Since $H$ is Haagerup, there is an $H$-invariant measured walls structure $(H,\sigma)$ whose associated pseudo-distance is proper. Consider the measured walls structure $(W\rtimes G,\hat{\sigma})$ from Lemma \ref{aux}. For any finite subsets $F,F'$ of $G/L$ and $G$, the pseudo-distance $d_{\hat{\sigma}}$ is proper in restriction to the subset
$\{wg\in W\rtimes G\,|\,\Supp(w)\subset F,g\in F'\}$ of $W\times G$.

We finally get that the measured walls structure $(W\rtimes G,\rho^*\tilde\mu+p^*\lambda+\hat{\sigma})$ is $W\rtimes G$-invariant and the corresponding distance is proper.
\end{proof}

\begin{rem}
It readily follows from the proof that\begin{itemize}
\item if $L$ is finite (e.g. in the case of standard wreath products, in which $\rho$ is the identity), then it is enough to consider $\rho^*\tilde\mu+\hat{\sigma}$;
\item if $H$ is finite, then it is enough to consider $\rho^*\tilde\mu+p^*\lambda$, although the distance associated to $\hat{\sigma}$ may still be unbounded;
\item in particular, when $L$ and $H$ are both finite, then the pseudo-distance $d_{\rho^*\tilde\mu}$ is proper.
\end{itemize}
\end{rem}

\subsection{Relative Property T}
If $G$ is a group and $X$ a subset, recall from \cite{C} that $(G,X)$ has relative Property T if whenever a net of positive definite functions converges pointwise to 1, the convergence is uniform on $X$. It is known that this holds if and only if $G$ has some finitely generated subgroup $H$ containing $X$ such that $(H,X)$ has relative Property T \cite[Theorem 2.5.2]{C}. Moreover, if $G$ is countable, then $(G,X)$ has relative Property T if and only if any conditionally negative definite function on $G$ is bounded on $X$ \cite[Theorem 2.2.3]{C}, or equivalently if any measure definite function on $G$ is bounded on $X$. Here is a reformulation of Theorem \ref{wreathT}.

\begin{thm}\label{wreathTbis} Let $H,G$ be any groups. Let $C$ be a subset of the standard wreath product $H\wr G$. Then $(H\wr G,C)$ has relative Property T if and only if the four following conditions are fulfilled:
\begin{itemize}
\item $(G,C_1)$ has relative Property T, where $C_1$ is the projection of $C$ on $G$;
\item $(H,C_2)$ has relative Property T, where $C_2$ is the projection of $C$ on $H$, i.e. the union of all projections on all copies of $H$;
\item $(G,\Supp(C))$ has relative Property T, where $\Supp(C)$ is the union of supports of all $w$, for $wg\in C\subset H^{(G)}\times G$;
\item The function $wg\mapsto\#\Supp(w)$ is bounded on $C$.
\end{itemize}
\end{thm}

To get Theorem \ref{wreathT}, simply put $Y=C_2$ and $Z=C_1\cup\Supp(C)$.

\begin{proof} Using the remarks above, we can suppose that $H$ and $G$ are countable.

If $\mu$ and $\sigma$ are invariant measured walls structures on $G$ and $H$, we freely use the notation used in the proof of Theorem \ref{corhagper}. First, we prove that the conditions of the theorem are necessary:

\begin{itemize}\item If $(H\wr G,C)$ has relative Property T, then we see by projecting that $(G,C_1)$ has relative Property T.

\item If the function $wg\mapsto\#\Supp(w)$ is unbounded on $C$, we pick $\sigma$ as the discrete measured walls structure on $H$, that is $\sigma(A)=\frac 12 \#\{h\in H|\{h\}\in A\}$, so that $d_\sigma(h,h')=0$ if $h=h'$ and $1$ otherwise. Then $d_{\hat{\sigma}}(w,1)=\#\Supp(w)$ is unbounded on $C$.

\item If $(H,C_2)$ does not have relative Property T, then we pick $\sigma$ unbounded on $C_2$ and get that $d_{\hat{\sigma}}$ is unbounded on $C$.

\item Finally if $(G,\Supp(C))$ does not have Property T, pick $\mu$ with $d_\mu$ unbounded on $\Supp(C)$, so that $d_{\tilde{\mu}}$ is unbounded on $C$.
\end{itemize}
Let us now prove that the conditions are sufficient. Suppose they all hold. Let $\psi$ be a measure-definite function on $H\wr G$. Note that $\psi$ is sub-additive. Identify $H$ with the subgroup of $H^{(G)}$ consisting of functions supported at the identity of $G$. There exists a positive constant $K$ such that: (i) $\psi(g)\leq K$ for all $g\in C_1$; (ii) $\psi(h) \leq K$ for all $h\in C_2$; (iii) $\psi(g)\leq K$ for all $g\in \Supp(C)$; (iv) $\#\Supp(w)\leq K$ for all $wg\in C$. Hence, any $wg\in C$ can be written in the form
\[
 wg = g_1h_1g_1^{-1} \cdots g_k h_k g_k^{-1}\cdot g
\]
with $\Supp(w) = \{g_1, \ldots, g_k\}$, whence $k\leq K$, and $h_1,\ldots, h_k\in C_2$. Recall that $g\in C_1$. Consequently, one has $\psi(wg)\leq (3k+1) K \leq 3K^2 + K$; this proves that $\psi$ is bounded on $C$.
\end{proof}

In the permutational case, we can characterize the relative Property T for certain {\it subgroups}.

\begin{prop}\label{subgrinwreath} Let $H, G$ be groups, $X$ a non-empty 
$G$-set, and $W=H^{(X)}$. For $x\in X$, let $p_x:W\rightarrow H:w\mapsto w_x$
denote the projection on the $x$-th factor of $W$. Let $K$ be a subgroup of 
$W$. The following are equivalent:
\begin{enumerate}
\item[(i)] $(W,K)$ has relative Property (T);
\item[(ii)] $(H\wr_XG,K)$ has relative Property T;
\item[(iii)] there exists a finite subset $C\subset X$ such that $K\subset 
H^C$, and $(H,p_x(K))$ has relative Property T for every $x\in C$.
\end{enumerate}
\end{prop}

\begin{lem}\label{combi}
Let $X$ be an infinite set and $(S_i)_{i\in I}$ a family of finite subsets of
bounded cardinality. Suppose that $\bigcup_i S_i$ is infinite. Then there
exists an infinite subset $J\subset I$, a subset $F\subset X$, and pairwise 
disjoint nonempty finite subsets $F_j\subset X-F$ ($j\in J$) such that 
$S_j= F\cup F_j$ for all $j\in J$.
\end{lem}
\begin{proof}
First, let $(i_n)$ be a sequence in $I$ such that $S_{i_n}$ is not contained 
in $\bigcup_{m<n}S_{i_m}$. Define $J_0=\{i_n:n\ge 0\}$. 
Let $J_1\subset J_0$ be an infinite subset such that the cardinality $d$ of 
$F=\underline{\lim}_{j\in J_1}S_j$ is maximal and set $F_j=S_j-F$ for 
$j\in J_1$. Note that $F_j$ is empty, that is $S_j$ is contained 
in $F$, only for finitely many $j$'s because of the definition of $J_0$. 
Moreover, for any $j_1\in J_1$, the set of indices $j\in J_1$ such that 
$F_{j_1}\cap F_j\neq \emptyset$ is finite, since otherwise one 
would get an infinite subset
$J_2\subset J_1$ contradicting the maximality of $d$. Therefore there 
exists $J\subset J_1$  fulfilling the desired properties.
\end{proof}

\begin{proof}[Proof of Proposition \ref{subgrinwreath}]
$(i)\Rightarrow (ii)$ is trivial. For $(ii)\Rightarrow(iii)$, assume that $(H\wr_XG,K)$ has relative Property T. Exactly as in the proof of Theorem \ref{wreathTbis}, it is seen that the function $w\mapsto\#\Supp(w)$ is bounded on $K$. 
Let us check that $C:=\bigcup_{w\in K}\Supp(w)$ is finite. Otherwise, apply Lemma \ref{combi} to find in $K$ a sequence $(w_n)$ with support $F\cup F_n$ with $F,(F_n)$ pairwise disjoint and $F_n$ not empty. Clearly, the elements $\prod_{i=1}^nw_i$ of $K$ have support of unbounded cardinality, a contradiction. 
If, for some $x\in C$, the pair $(H,p_x(K))$ does not have the relative Property T, we find a measured wall structure $\sigma$ on $H$ such that $d_\sigma$ is unbounded on $p_x(K)$, and get that $d_{\hat{\sigma}}$ is unbounded on $K$, hence contradicting our assumption.

Finally, assume that the conditions in $(iii)$ are satisfied. Then $(H^C,\prod_{x\in C}p_x(K))$ has the relative Property T. So increasing $H^C$ to $W$ and decreasing $\prod_{x\in C}p_x(K)$ to $K$, we see that $(W,K)$ has the relative Property T, i.e. $(i)$ is satisfied.
\end{proof}

As a special case of Proposition \ref{subgrinwreath}, the pair $(H\wr_XG,W)$ has relative Property T if and only if $W$ has Property T, if and only if $X$ is finite and $H$ has Property T (this can also be deduced from Section 2.8 in \cite{BHV}).

Proposition \ref{subgrinwreath} allows us to characterize the existence of an infinite subgroup with relative Property T. This improves on a result of Neuhauser \cite[Theorem~1.1]{Neu}.

\begin{thm}\label{relTinpermwreath}
Let $H,G$ be groups, and $X$ a non-empty $G$-set. Then $H\wr_X G$ has relative Property T with respect to some infinite subgroup if and only if either $H$ or $G$ has relative Property T with respect to some infinite subgroup.
\end{thm}
\begin{proof}
The condition is obviously sufficient. Conversely suppose that $G$ and $H$ do not have relative Property T with respect to any infinite subgroup, and let us show that the same holds for $H\wr_X G$. So let $K$ be a subgroup of $H\wr_XG$ with the relative Property T. Then the projection of $K$ on $G$ is finite. Replacing $K$ with a finite index subgroup, we may assume that $K\subset H^{(X)}$. By Proposition \ref{subgrinwreath}, there exists a finite subset $C\subset X$, with $K\subset H^C$, and moreover each projection $p_x(K)$ (for $x\in C$) is finite. So $K$ itself is finite.
\end{proof}

\section{Wreath product with a free group}\label{sct:comp}

Given a tree $T=(V,E)$, let us recall that each (unoriented) edge $e$ defines a partition $V = V_e^+ \sqcup V_e^-$ corresponding to connected components of $T\setminus\{e\}$. Setting $\mathcal W = \bigcup_{e\in E}\{V_e ^+, V_e^-\}$ and $\mu=\frac 12 \sum_{W\in\mathcal W}\delta_W$, we obtain a discrete walls structure on $V$ such that the associated distance $d_\mu$ coincides with the tree distance.

Let us now consider a finitely generated free group $\F = \F(S)$. Its Cayley graph being a tree, we obtain a discrete wall structure $\mu$ on $\F$, whose distance coincides with the word length with respect to $S$. Obviously, $\mu$ is $\F$-invariant.

Let $H$ be a finitely generated group with word length $|\cdot|$ with respect to some generating subset $S'$. We form the wreath product $H\wr \F = H^{(\F)}\rtimes \F$. In what follows, we identify $\F$ with the subgroup $\{1\}\times \F$ and $H$ with the functions supported on $\{1\}$ in $H^{(\F)}$. With this convention, $S\cup S'$ is a finite generating set of $H\wr \F$. The associated word length is given by
$$|wg|=m(w,g) + \sum_{g'\in \F}|w_{g'}|,$$
where $g\in \F$, $w\in H^{(\F)}$ and $m(w,g)$ is the length of the shortest path in the Cayley graph of $\F$ joining $1$ to $g$ and covering $\Supp(w)$ (see for instance \cite[Theorem 1.2]{Par}).

We can now introduce the main result of this section.
\begin{defn}
 Let $G$ be a finitely generated group with word length $|\cdot|$ and let $d$ be a left-invariant pseudo-distance on $G$. A \emph{compression function} for $d$ is a function $\alpha:\R_+\to \R_+$
 such that $d(1,g)\ge\alpha(|g|)$ for all $g\in G$.
\end{defn}
\begin{prop}\label{fctComp}
Let $\alpha:\R_+\to \R_+$ be a non-decreasing and subadditive function and let $\F$ be a finitely generated free group. If a finitely generated group $H$ has a left-invariant measured walls structure whose corresponding pseudo-distance admits $\alpha$ as a compression function, then so does the wreath product $H\wr \F$.
\end{prop}
\begin{proof} Choose a constant $C\ge 1$ such that the function $\beta = \alpha/C$ satisfies $\beta(r) \le r/2$ for all $r\in \N^*$.
Let $\sigma$ be a left-invariant measured walls structure on $H$ such that $\alpha$ is a compression function for $d_\sigma$. Then $\beta$ is a compression function for $d_\sigma$, which is still non-decreasing and subadditive.

Let $\Phi(\sigma) = \tilde{\mu}+\hat{\sigma}$, where $\hat \sigma$ is constructed as in Lemma \ref{aux}, with $Y=\F$, and $\tilde \mu$ is constructed as in Section \ref{sec:construction} with respect to the gauge $\phi:H^{(\F)} \times H^{(\F)} \to 2^{(\F)}$ given by $\phi(w,w') = \Supp(w^{-1}w')$. This is a left-invariant measured walls structure on $H\wr \F$ by Theorem \ref{cons} and Lemma \ref{aux}. We also get the formula
$$d_{\Phi(\sigma)}(wg,1)=\#\big\{\text{edges cutting }\Supp(w)\cup\{1,g\}\big\} +\sum_{g'\in \F}d_\sigma(1,w_{g'}),$$
where we identify the free group $\F$ to the vertices of its Cayley tree.

Let us now recall that the edges cutting $\Supp(w)\cup\{1,g\}$ form a finite subtree, namely the convex hull of $\Supp(w)\cup\{1,g\}$.

\begin{claa}
Let $T$ be a finite tree with $n$ edges. Then there exists a loop of length $2n$ covering $T$.\label{convtree}
\end{claa}
This is proved by a trivial induction: if $T$ has at least one edge, pick a terminal vertex $v$, and use the induction hypothesis on the subtree with vertex set $T-\{v\}$.
\medskip

It follows from the claim and properties of $\beta$ that
\begin{eqnarray*}
 d_{\Phi(\sigma)}(wg,1) & \ge & \frac12m(w,g)+\sum_{g'\in \F}\beta(|w_{g'}|) \\
 & \ge & \beta(m(w,g)) + \sum_{g'\in \F}\beta(|w_{g'}|)   \\
 & \ge & \beta\left(m(w,g)+\sum_{g'\in \F}|w_{g'}|\right)=\beta(|wg|).
\end{eqnarray*}
Hence, $\beta$ is a compression function for $d_{\Phi(\sigma)}$, and $\alpha = C\beta$ is a compression function for $d_{C\cdot\Phi(\sigma)}$.
\end{proof}

In particular, Proposition \ref{fctComp} can be applied to the $n$-fold iterated wreath product $H_n=(\dots((H\wr \F_{k_1})\wr \F_{k_2})\dots \wr \F_{k_n})$. For instance, if $H$ has a left-invariant measured walls structure whose corresponding pseudo-distance has compression $n^d$ for some $d\in [0,1]$, then so does $H_n$.
\begin{defn}
 Let $G$ be a finitely generated group. We set $B_1(G)$ to be the supremum of the numbers $d\in [0,1]$ such that $r\mapsto r^d$ is a compression function for some pseudo-distance associated to a left-invariant measured walls structure on $G$.
\end{defn}

Using $(iii)\Leftrightarrow (iv)$ in Proposition \ref{equiv}, one can see that $B_1(G)$ is equal to the {\em (strong) equivariant $L^1$-compression} of $G$, that is, the supremum of the numbers $d\in [0,1]$ such that there exists an $L^1$-space $E$, endowed with an isometric $G$-action, and a $G$-equivariant map $f:G\to E$ which satisfies
\[
 |x^{-1}y|^d \le ||f(x) - f(y)||  \quad \forall x,y\in G
\]
(note that such a map always satisfies $||f(x) - f(y)|| \le C\cdot |x^{-1}y|$ for some positive constant $C$). Proposition \ref{fctComp} has the following immediate consequence:

\begin{cor}
 One has $B_1(H_n)=B_1(H)$ for all $n\ge 1$.
\end{cor}

\begin{rem}One could define another notion of equivariant $L^1$-compression by replacing $E$ by a metric space $Y$, endowed with an isometric $G$-action, which is isometrically embeddable into an $L^1$-space. Note that these numbers would be equal if we could prove the ``missing implication'' in Proposition \ref{equiv}.
\end{rem}

\begin{rem}
It might be tempting to streamline the construction of Section \ref{sec:construction} and the proof of Theorem \ref{corhagper} by directly constructing a measure definite kernel instead of a left-invariant measured walls structure. However, since we do not know if the ``missing implication'' in Proposition \ref{equiv} holds (see the remark following this proposition), and since the construction as well as $d_{\tilde{\mu}}$ itself definitely depends on the measured walls structure and not only on the associated distance, there would be a loss in the iterates, for instance we would only get $B_1(H_n)\ge 2^{-n}B_1(H)$.
\end{rem}

\appendix
\section{Measured walls structures vs. spaces with measured walls \`a la Cherix-Martin-Valette}\label{HSMeasure}

For the purpose of the appendix, we introduce the following definition. It seems even more natural than the one of measured walls structure, but we chose the latter because it is sometimes more tedious to work with partitions (walls) rather than subsets (half-spaces).

\begin{defn}\label{MesWalStrAlt}
An {\it alternate measured walls structure} is a pair $(X,\nu)$ where $X$ is a set and $\nu$ is a Borel measure on the set of bipartitions of $X$ such that for all $x,y\in X$
$$d_\nu(x,y):=\nu(\{A|A\vdash\{x,y\}\})<\infty.$$
\end{defn}

Define $\mathcal{W}(X)$, the set of bipartitions of $X$, i.e. the quotient of $2^X$ by the complementation involution $s:2^X\rightarrow 2^X;A \mapsto A^c$. Besides, say that a measured walls structure on $X$ is symmetric if
$\mu$ is $s$-invariant. Let $p$ be the canonical map $2^X\to\mathcal{W}(X)$.

\begin{lem}\label{lem:rev2}
The map $\mu\mapsto p_*\mu$ is
a (canonical) bijection between the set of symmetric measured walls structures on $X$ and the set of alternate measured walls structures on $X$, which preserves the associated pseudo-distance.
\end{lem}
\begin{proof}
This map obviously preserves the associated pseudo-distance.

Fix $x\in X$. The restriction of $p$ to $\{A\in 2^X|A\ni x\}$ is a homeomorphism whose inverse we denote by $j$. Note that $j_*\nu$ is not symmetric. We have $p\circ j=\text{Id}$ and $p\circ s\circ j=\text{Id}$, so $p_*j_*\nu=\nu$ and $p_*s_*j_*\nu=\nu$. Now $$T\nu=\frac{1}{2}(j_*\nu+s_*j_*\nu)$$ is symmetric, and $p_*T\nu=\nu$. This proves that $T$ is a right inverse of $p_*$. We claim that, if we restrict to symmetric measured walls structures, this is an inverse.

Consider $B$ a Borel subset of $2^X$ and write $B=B_x\sqcup B'_x$, where $B_x=\{A\in B|A\ni x\}$. Then $$j_* p_*\mu(B)=j_* p_*\mu(B_x)+j_* p_*\mu(B'_x)=\mu(B_x)+\mu(s(B_x)) + 0$$
and similarly
$$s_*j_* p_*\mu(B)= 0 + \mu(s(B'_x))+\mu(B'_x),$$
so
$$Tp_*\mu(B)=\frac{1}{2}(\mu(B_x)+\mu(s(B_x))+\mu(s(B'_x))+\mu(B'_x)).$$
If $\mu$ is symmetric, this gives
\[
Tp_*\mu(B)=\mu(B_x)+\mu(B'_x)=\mu(B).\qedhere
\]
\end{proof}

Note that any alternate measured walls structure defines a space with measured walls, with Borel subsets as $\sigma$-algebra.

Conversely, if $(X,\Wc,\Bc,\mu)$ is a space with measured walls (as in \cite[Definition~2]{CMV}) on a countable set $X$, consider the embedding $i$ of $\Wc$ into $\Wc(X)$. Then the $\sigma$-algebra $i_*\Bc=\{A\in\Wc(X)|A\cap\Wc\in\Bc\}$ contains all basic clopen sets $\mathcal{A}_{\{x,y\}}=\{A|A\vdash\{x,y\}\}$, hence contains all Borel sets, and $i_*\mu$ provides an alternate measured walls structure. This corresponds to a measured walls structure by Lemma \ref{lem:rev2}.

As all the constructions given here are canonical (the introduction of $x$ in the proof of Lemma \ref{lem:rev2} is not canonical but is only used to prove that some canonically defined map is a bijection), they are compatible with group actions.

\section{Hecke pairs}\label{AppHecke}

\begin{defn}\label{defHecke} Let $H$ be a subgroup of the group $G$. The pair $(G,H)$ is a Hecke pair if all $H$-orbits on $G/H$ are finite.
\end{defn}

Clearly, $(G,H)$ is a Hecke pair if either $H$ is finite, or $H$ has finite index, or $H$ is normal in $G$. A non-trivial example is the pair $(SL_2(\Q),SL_2(\Z))$ (see e.g. \cite{Krieg}, to which we also refer for more background). The following result allows us to construct many more examples.

\begin{prop}\label{CarHecke} For a subgroup $H$ of a group $G$, consider the following properties:
\begin{enumerate}
\item[i)] The homogeneous space $G/H$ carries a $G$-invariant structure of a connected, locally finite graph.
\item[ii)] There exists a $G$-invariant, proper, non-negative kernel on $G/H$.
\item[iii)] $(G,H)$ is a Hecke pair.
\end{enumerate}
Then $(i)\Rightarrow (ii)\Leftrightarrow (iii)$. If moreover $G$ is finitely generated, then all conditions are equivalent.
\end{prop}

\begin{proof}
$(i)\Rightarrow (ii)$ If $G/H$ carries a $G$-invariant structure of a connected, locally finite graph $X$, then the distance function on $X$ is a $G$-invariant proper kernel on $G/H$ (properness following from the finiteness of the balls in $X$).

$(ii)\Rightarrow(iii)$ Let $K$ be a proper $G$-invariant kernel on $G/H$. Fix $gH\in G/H$. Set $R=K(eH,gH)$ and $F=\{xH\in G/H:K(eH,xH)= R\}$; the latter is a finite set, by properness of $K$. For $h\in H$, we have by $G$-invariance of $K$:
$$K(eH,hgH)=K(h^{-1}H,gH)=K(eH,gH)=R,$$
so that $hgH\in F$. This shows that $H$-orbits in $G/H$ are finite.

$(iii)\Rightarrow(ii)$ Assume that $(G,H)$ is a Hecke pair. Let $f:G\rightarrow\N$ be a proper function; replacing $f$ by $f+\check{f}$, we may assume that $f$ is symmetric. Define, for $g\in G$:
$$k(g)=\min\{f(w):w\in HgH \}.$$
Then $k:G\rightarrow\N$ is a symmetric, bi-$H$-invariant function, so $K(gH,g'H)=k(g^{-1}g')$ is a well-defined $G$-invariant, symmetric, kernel on $G/H$. To check properness, fix $R\geq 0$. If $k(g)\leq R$, then the double class $HgH$ meets the finite set $F_R=f^{-1}[0,R]$, so that $gH\subset HF_RH$. As $(G,H)$ is a Hecke pair, $HF_RH$ consists of finitely many left cosets of $H$, so that there are finitely many $gH$'s with $K(eH,gH)\leq R$.

Assume now that $G$ is generated by a finite symmetric set $S$, and that $(G,H)$ is a Hecke pair. We define a graph $X$ with vertex set $G/H$ by declaring $gH,\,g'H$ to be adjacent if $gHs\cap g'H\neq\emptyset$ for some $s\in S$. This adjacency relation is symmetric, as $S=S^{-1}$, and obviously $G$-invariant. Observe now that $gHs\cap g'H\neq\emptyset$ if and only if $g'H\subset gHsH$. This holds in particular if $g'=gs$. So the canonical projection $G\rightarrow G/H; g\mapsto gH$ induces a homomorphism of graphs from the Cayley graph ${\mathcal G}(G,S)$ to $X$. In particular, $X$ is connected. Finally, fix $gH\in G/H$; since $gHsH$ consists of finitely many left cosets of $H$, and $S$ is finite, we see that there are finitely many cosets $g'H$'s such that $g'H\subset \bigcup_{s\in S}gHsH$, so that $gH$ has finitely many neighbors, and $X$ is locally finite. This proves $(iii)\Rightarrow(i)$.

\end{proof}

\baselineskip=16pt

\end{document}